\documentclass[a4paper,11pt]{article}
\usepackage{amsmath, amsfonts, amscd, amssymb, amsthm, enumerate, xypic}

\def\car{\text{char}}
\def\Mat{\text{M}}

\def\defterm{\textbf}

\newcommand{\tr}{\operatorname{tr}}
\newcommand{\Ker}{\operatorname{Ker}}

\renewcommand{\setminus}{\smallsetminus}

\def\F{\mathbb{F}}
\def\K{\mathbb{K}}

\def\N{\mathbb{N}}

\def\calA{\mathcal{A}}

\def\calR{\mathcal{R}}

\def\lcro{\mathopen{[\![}}
\def\rcro{\mathclose{]\!]}}

\theoremstyle{definition}
\newtheorem{Def}{Definition}
\newtheorem{Not}[Def]{Notation}

\theoremstyle{plain}
\newtheorem{theo}{Theorem}
\newtheorem{prop}[theo]{Proposition}
\newtheorem{cor}[theo]{Corollary}
\newtheorem{lemme}[theo]{Lemma}

\theoremstyle{remark}
\newtheorem{Rems}{Remarks}
\newtheorem{Rem}[Rems]{Remark}

\title{On sums of idempotent matrices over a field of positive characteristic}
\author{Cl\'ement de Seguins Pazzis
\footnote{Professor of Mathematics at Lyc\'ee Priv\'e Sainte-Genevi\`eve, 2, rue
de l'\'Ecole des Postes, 78029 Versailles Cedex, FRANCE.}
\footnote{e-mail address: dsp.prof@gmail.com}}
\date{\today}

\begin{document}

\maketitle

\begin{abstract}
We study which square matrices are sums of idempotents over a field
of positive characteristic; in particular, we prove that any such matrix, provided it is large enough,
is actually a sum of five idempotents, and even of four when the field is a prime one.
\end{abstract}



\section{Introduction}

In this article, $\K$ will denote a field of characteristic $\car(\K)=p \neq 0$.
The prime subfield of $\K$ is then isomorphic to $\F_p$, so
we can assume, without loss of generality, that it is precisely $\F_p$.
We choose an algebraic closure $\overline{\K}$ of $\K$. We will use the French convention for the set of integers:
$\N$ will denote the set of non-negative integers, and $\N^*$ the one of positive integers.

\vskip 2mm
\noindent
An idempotent matrix of $\Mat_n(\K)$ is a matrix $P$
verifying $P^2=P$, i.e. idempotent matrices represent projectors in finite dimensional
vector spaces. Of course, any matrix similar to an idempotent is itself an idempotent.

\vskip 5mm
In recent history, decomposition of matrices into sums of idempotents have been extensively studied over fields of characteristic $0$. In this paper, we wish to determine:
\begin{enumerate}[(i)]
\item Which matrices of $\Mat_n(\K)$ are sums of idempotents?
\item What is the lowest integer $s_n(\K)$ such that every matrix of $\Mat_n(\K)$
which is a sum of idempotents can actually be decomposed as a sum of $s_n(\K)$ idempotents?
\end{enumerate}
The first question will be easily answered in section 3 (the trace says it all \ldots), but the second is in general a very hard one.
We will nevertheless determine $s_n(\K)$ for small fields and fields of small characteristic,
give good lower and upper bounds for $s_n(\K)$ in the general case, and
actually calculate $s_n(\K)$ for large $n$.
In order to do so, we will need a few technical results on cyclic matrices, which
we have reviewed in section 4. We will start by reviewing classic results of Hartwig, Putcha
and the author on sums and differences of idempotents in a matrix algebra (see \cite{HP} and \cite{dSPidem2}).

\section{Additional notations}

Given a list $(A_1,\dots,A_p)$ of square matrices, we will denote by
$$D(A_1,\dots,A_p):=\begin{bmatrix}
A_1 & 0 & & 0 \\
0 & A_2 & & \vdots \\
\vdots & & \ddots & \\
0 & \dots & &  A_p
\end{bmatrix}
$$
the block-diagonal matrix with diagonal blocks $A_1$, \dots, $A_p$.

\vskip 2mm
\noindent Similarity of two matrices $A$ and $B$ of $\Mat_n(\K)$ will be written $A\sim B$.

\vskip 2mm
\noindent
We denote by $H_{n,p}$ the elementary matrix
$\begin{bmatrix}
0 & \cdots & 0 &  1 \\
\vdots & & \vdots &  \\
0 & \cdots & 0 & 0
\end{bmatrix} \in \Mat_{n,p}(\K)$ with only non-zero coefficient located on the first row and
$p$-th column. \\
For $k \in \N^*$, we set
$$F_k:=D(0,\dots,0,1)\in \Mat_k(\K).$$

\section{Sums and differences of two idempotents}

\begin{Def}
Let $\calA$ be a $\K$-algebra and
$(\alpha_1,\dots,\alpha_n)\in (\K^*)^n$.
An element $x \in \calA$ will be called an
\textbf{$(\alpha_1,\dots,\alpha_n)$-composite} when there are idempotents $p_1,\dots,p_n$
such that $x=\underset{k=1}{\overset{n}{\sum}}\alpha_k.p_k$. \\
\end{Def}

\begin{Not}
When $A$ is a matrix of $\Mat_n(\K)$, $\lambda \in \overline{\K}$ and $k \in \N^*$, we denote by
$$n_k(A,\lambda):=\dim \Ker (A-\lambda.I_n)^k-\dim \Ker (A-\lambda.I_n)^{k-1},$$
i.e. $n_k(A,\lambda)$ is the number of  size greater or equal to $k$
for the eigenvalue $\lambda$ in the Jordan reduction of $A$
(in particular, it is zero when $\lambda$ is not an eigenvalue of $A$).
We also denote by $j_k(A,\lambda)$ the number of blocks of size $k$ for the eigenvalue $\lambda$ in the Jordan reduction of $A$.
\end{Not}

\begin{Def}
Two sequences $(u_k)_{k \geq 1}$ and $(v_k)_{k \geq 1}$ are said to be \defterm{intertwined}
when:
$$\forall k \in \N^*, \; v_k \leq u_{k+1} \quad \text{and} \quad u_k \leq v_{k+1.}$$
\end{Def}

With that in mind, the problem of determining whether a particular matrix $A \in \Mat_n(\K)$
is a $(1,-1)$-composite or a $(1,1)$-composite is completely answered by the following theorems, proved
in \cite{HP} and \cite{dSPidem2}.

\begin{theo}\label{HPdiff}
Assume $\car(\K) \neq 2$ and let $A \in \Mat_n(\K)$.
Then $A$ is a $(1,-1)$-composite iff all the following conditions hold:
\begin{enumerate}[(i)]
\item The sequences $(n_k(A,1))_{k \geq 1}$ and $(n_k(A,-1))_{k \geq 1}$ are intertwined.
\item $\forall \lambda \in \overline{\K} \setminus \{0,1,-1\}, \; \forall k \in \N^*, \; j_k(A,1)=j_k(A,-1)$.
\end{enumerate}
In particular, every nilpotent matrix is a difference of idempotents.
\end{theo}

\begin{theo}\label{HPsum}
Assume $\car(\K) \neq 2$, and let $A \in \Mat_n(\K)$.
Then $A$ is a $(1,1)$-composite iff all the following conditions hold:
\begin{enumerate}[(i)]
\item The sequences $(n_k(A,0))_{k \geq 1}$ and $(n_k(A,2))_{k \geq 1}$
are intertwined.
\item $\forall \lambda \in \overline{\K} \setminus \{0,1,2\}, \; \forall k \in \N^*, \;
j_k(A,\lambda)=j_k(A,2-\lambda)$.
\end{enumerate}
\end{theo}

\begin{theo}\label{HPsumcar2}
Assume $\car(\K)=2$ and let $A \in \Mat_n(\K)$.
Then $A$ is a $(1,-1)$-composite iff
for every $\lambda \in \overline{\K} \setminus \{0,1\}$,
all blocks in the Jordan reduction of $A$ with respect to $\lambda$ have an even size. \\
In particular, every triangularizable matrix with eigenvalues in $\{0,1\}$ is a sum (and a difference) of two idempotents.
\end{theo}

\section{When is a matrix of $\Mat_n(\K)$ a sum of idempotents?}

\begin{theo} ${}$ \\
A matrix $A \in \Mat_n(\K)$ is a sum of idempotents iff $\tr A \in \F_p$. \\
In particular, every matrix of $\Mat_n(\F_p)$ is a sum of idempotents.
\end{theo}

\begin{proof}
The ``only if" part is clear because an idempotent of rank $r$ in $\Mat_n(\K)$
has trace $r.1_\K \in \F_p$. \\
Conversely, let us first remark that any nilpotent matrix $N$ is
a sum of idempotents: indeed, by Proposition 1 of \cite{HP}, there are idempotents $Q_1$ and $Q_2$
such that $Q_1-Q_2=N$, so $N=Q_1+(p-1).Q_2$. \\
Assume $\tr A \in \F_p$, and choose $k \in \N$ such that $\tr A=k.1_\K$. \\
Let us choose an idempotent $Q \in \Mat_n(\K)$ of rank $1$, and
set $B:=A-k.Q$, so $\tr B=0$. It suffices to prove that $B$ is itself
a sum of idempotents. Since this is trivial when $B=0$, we now assume $B \neq 0$.
\begin{itemize}
\item \emph{The case $B$ is not scalar.} Then (cf. \cite{Fillmore}) $B$ is similar to a matrix
$C$ with diagonal coefficients all equal to zero;
such a $C$ can thus be written as the sum of a strictly upper triangular matrix
and a strictly lower triangular matrix, each of which is nilpotent.
Therefore, $B$ is a sum of idempotents.
\item \emph{The case $B$ is scalar.} Since $B \neq 0$, we must have $n \geq 2$,
so we can choose a non-zero nilpotent $N \in \Mat_n(\K)$. Hence
$B-N$ is not scalar and satisfies the conditions of the first case, so it is
a sum of idempotents.
Therefore, $B=(B-N)+N$ is a sum of idempotents.
\end{itemize}
In all cases, $B$ is a sum of idempotents, which finishes our proof.
\end{proof}

A closer inspection at the previous proof shows that any matrix of
$\Mat_n(\K)$ with trace in $\F_p$ is a sum of at most $4\,p$ idempotents.
In the rest of our paper, we will try to find a much tighter upper bound.

\section{A review of cyclic matrices}

The characteristic polynomial of a matrix $M$ will be denoted by $\chi_M$.

\vskip 2mm
\noindent
Let $P=X^n-\underset{k=0}{\overset{n-1}{\sum}}a_kX^k \in \K[X]$ be a monic polynomial with degree $n$.
Its \defterm{companion matrix} is
$$C(P):=\begin{bmatrix}
0 &   & & 0 & a_0 \\
1 & 0 & &   & a_1 \\
0 & \ddots & \ddots & & \vdots \\
\vdots & & & 0 & a_{n-2} \\
0 & & &  1 & a_{n-1}
\end{bmatrix}.$$
Its characteristic polynomial is precisely $P$, and so is its minimal polynomial.
We will set $\tr P:=\tr C(P)=a_{n-1}$ and $\deg P:=n$ (the degree of $P$).
We will use repeatedly the following basic fact
(cf. \cite{Gantmacher}):
when $P$ and $Q$ denote two mutually prime monic polynomials, one has
$$C(P\,Q) \sim \begin{bmatrix}
C(P) & 0 \\
0 & C(Q)
\end{bmatrix}.$$
\noindent
Let $A \in \Mat_n(\K)$. We say that $A$ is \defterm{cyclic} when
$A \sim C(P)$ for some polynomial $P$ (and then $P=\chi_A$).
A \defterm{good cyclic} matrix is a matrix of the form
$$A=\begin{bmatrix}
a_{1,1} & a_{1,2}  & &  & a_{1,n} \\
1 & a_{2,2} & &   &  \\
0 & \ddots & \ddots & & \vdots \\
\vdots & & & a_{n-1,n-1} & a_{n-1,n} \\
0 & & &  1 & a_{n,n}
\end{bmatrix}$$
with no condition on the $a_{i,j}$'s for $j \geq i$.

\vskip 2mm
\noindent This last lemma has been proven in \cite{dSPLC3} and
is the key to some of the results featured here:

\begin{lemme}[Choice of polynomial lemma]\label{cyclicfit}
Let $A \in \Mat_n(\K)$ and $B \in \Mat_p(\K)$ denote two good cyclic matrices, and
$P$ denote a monic polynomial of degree $n+p$ such that $\tr P=\tr A+\tr B$. \\
Then there exists a matrix $D \in \Mat_{n,p}(\K)$ such that
$$\begin{bmatrix}
A & D \\
H_{p,n} & B
\end{bmatrix} \sim C(P).$$
\end{lemme}

\section{General results on minimal decompositions}

\begin{Not}
For $n \in \N^*$, we let
$s_n(\K)$ denote the lowest integer $N$ such that every matrix
$A \in \Mat_n(\K)$ with $\tr A \in \F_p$ is a sum of $N$ idempotents.
\end{Not}

\noindent A lower bound for $s_n(\K)$ can easily found using the trace:

\begin{prop}\label{minoration}
For every integer $n \geq 1$, one has:
$$s_n(\K) \geq \frac{p-1}{n}$$
and equality cannot hold if $n>1$.
\end{prop}

\begin{proof}
Let $n \in \N^*$. Let $M \in \Mat_n(\K)$ such that $\tr M=(p-1).1_\K$. \\
Then $M$ is a sum of $s_n(\K)$ idempotents, each with a trace of the form $k.1_\K$
for some $k \in \lcro 0,n\rcro$, so $n\,s_n(\K) \geq p-1$.
If $n\,s_n(\K)=p-1$, then $M$ would be a sum of $s_n(\K)$ copies of $I_p$, so
it would be scalar. However, if $n \geq 2$, we can find a non-scalar $M \in \Mat_n(\K)$ such that
$\tr M=(p-1).1_\K$, so equality $n\,s_n(\K)=p-1$ cannot hold.
\end{proof}

\begin{theo}\label{sumof5}
For all $n \in \N \setminus \{0,1\}$, we have
$$s_n(\K) \leq 5+\Bigl[\frac{p-1}{n}\Bigr],$$
where $[x]$ denotes the greatest integer $k$ such that $k \leq x$. \\
In particular, if $n \geq p$, then every matrix of $\Mat_n(\K)$
with trace in $\F_p$ is a sum of five idempotents.
\end{theo}

\noindent Used in conjunction with Proposition \ref{minoration}, this yields:

\begin{cor}
For all $n \in \N \setminus \{0,1\}$,
$$1+\Bigl[\frac{p-1}{n}\Bigr] \leq s_n(\K) \leq 5+\Bigl[\frac{p-1}{n}\Bigr].$$
\end{cor}

\begin{proof}[Proof of Theorem \ref{sumof5}]
Let $n \in \N\setminus \{0,1\}$ and $A \in \Mat_n(\K)$ such that~$\tr A \in \F_p$.
The proof has two major steps:
\begin{enumerate}[(i)]
\item There are two idempotents $Q_1$ and $Q_2$ such that
$A-Q_1-Q_2$ is cyclic.
\item Every cyclic matrix of $\Mat_n(\K)$ with trace in $\F_p$
is a sum of $3+\bigl[\frac{p-1}{n}\bigr]$ idempotents.
\end{enumerate}
By reduction to a rational canonical form, we can find companion matrices
$C(P_1),\dots,C(P_N)$ such that
$$A \sim D\bigl(C(P_1),\dots,C(P_N)\bigr),$$
hence $A$ can be replaced with $A':=D\bigl(C(P_1),\dots,C(P_N)\bigr)$.
Let $n_k:=\deg P_k$ for all $k \in \lcro 1,N\rcro$.
Let then
$$Q_1:=\begin{bmatrix}
F_{n_1} & 0 & 0 & &   \\
-H_{n_2,n_1} & 0_{n_2} & \ddots  & &  \\
0 & 0 & F_{n_3} & \ddots &  & \\
\vdots & 0 & -H_{n_4,n_3} & 0_{n_4} & \ddots & \\
  &   &      0       &    & \ddots & \\
 & &  & \ddots &
\end{bmatrix} \in \Mat_n(\K)$$
and
$$Q_2:=\begin{bmatrix}
0_{n_1} & 0 & 0 & &   \\
0 & F_{n_2} & \ddots  & &  \\
0 & -H_{n_3,n_2} & 0_{n_3} & \ddots &  & \\
\vdots & 0 & 0 & F_{n_4} &  \ddots & \\
  &   &      & -H_{n_5,n_4} & 0_{n_5}      & \ddots & \\
 & &  &  & & \\
 & &   &     & &
\end{bmatrix} \in \Mat_n(\K).$$
Straightforward computation shows that $Q_1$ and $Q_2$ are idempotents,
and \\
$A-Q_1-Q_2$ is clearly a good cyclic matrix
with trace $\tr A-\tr Q_1-\tr Q_2 \in \F_p$. \\
It now remains to prove step (ii). \\
Let then $B \in \Mat_n(\K)$ be a cyclic matrix with trace $t \in \F_p$. \\
Without loss of generality, we may assume $B$ is a companion matrix. \\
It will of course suffice to prove that $B-I_n$
can be written as $-Q_0+\underset{k=1}{\overset{2+[(p-1)/k]}{\sum}}Q_k$
where the $Q_k$'s are idempotents. \\
Set $k \in \lcro 0,p-1\rcro$ such that $\tr (B-I_n) =k.1_\K$.
We can decompose
$k=a\,n+\ell$ for some $a \in \lcro 0,[(p-1)/n]\rcro$
and some $\ell \in \lcro 0,n-1\rcro$. \\
Set then $\ell':=\max(\ell,1)$, and let us decompose
$$B-(a+1).I_n=\begin{bmatrix}
B_1 & D_1 \\
H_{n-\ell',\ell'} & B_2
\end{bmatrix} \quad \text{with $B_1 \in \Mat_{\ell'}(\K)$, $B_2 \in \Mat_{n-\ell'}(\K)$ and
$D_1 \in \Mat_{\ell',n-\ell'}(\K)$.}$$
The matrices $B_1-I_{\ell'}$ and $B_2$ are good cyclic ones,
and
$$\tr(B_1-I_{\ell'})+\tr B_2=\tr B-(a\,n+\ell').1_\K \in \{0,-1\},$$
so Lemma \ref{cyclicfit} provides some $D_2\in \Mat_{\ell',n-\ell'}(\K)$ such that
$B':=\begin{bmatrix}
B_1-I_{\ell'} & D_2 \\
H_{n-\ell',\ell'} & B_2
\end{bmatrix}$
is similar to either $C(X^{n-1}(X+1))$ or $C(X^n)$.
In any case, Theorems \ref{HPdiff} and \ref{HPsumcar2} show that
$B'=B-(a+1).I_n-\begin{bmatrix}
I_{\ell'} & D_1-D_2 \\
0 & 0
\end{bmatrix}$ is a difference of two idempotents.
Since
$\begin{bmatrix}
I_{\ell'} & D_1-D_2 \\
0 & 0
\end{bmatrix}$ is an idempotent itself, we conclude that $B$ is a sum of $a+3$ idempotents, which
finishes our proof.
\end{proof}

\section{The case of $\F_2$ and $\F_3$}

\begin{prop}
Assume $\# \K \leq 3$. Then, for every $n \in \N^*$, every matrix of $\Mat_n(\K)$
is a sum of three idempotents.
\end{prop}

\begin{proof}
If $\# \K \leq 2$, then the previous theorem is an easy consequence of Theorem 1 of \cite{dSPLC3}, but we will
give here a more elementary proof. By reduction to the rational canonical form, it suffices
to prove that every \emph{cyclic} matrix of $\Mat_n(\K)$ is a sum of three idempotents. \\
Let then $P \in \K[X]$ be a monic irreducible polynomial of degree $m$.
Set $J:=(\delta_{i,j+1})_{1 \leq i,j \leq n}$, and let us write
$$C(P)=\begin{bmatrix}
J & C \\
H_{1,n-1} & \tr P
\end{bmatrix} \quad \text{with $C \in \Mat_{n-1,1}(\K)$.}$$
Set
$P_1:=(X-1)^{m-1}(X-\tr P+m.1_\K)$, so
$$C(P_1)=\begin{bmatrix}
J & C_1 \\
H_{1,n-1} & \tr P-1
\end{bmatrix} \quad \text{for some $C_1 \in \Mat_{n-1,1}(\K)$}$$
and $C(P_1)$ is a sum of two idempotents by Theorems \ref{HPsum} and \ref{HPsumcar2} since
$\# \K \leq 3$.
Finally
$$C(P)-C(P_1)=\begin{bmatrix}
0 & C-C_1 \\
0 & 1
\end{bmatrix}$$
is an idempotent, so $C(P)$ is a sum of three idempotents.
\end{proof}

The previous result fails for fields with at least $4$ elements, even
if we only consider matrices with trace in $\F_p$:

\begin{enumerate}[(i)]
\item Assume $\K \simeq \F_p$ for some prime $p \geq 5$. \\
Then $(p-1).I_n$ is not a sum of three idempotents. Indeed,
for any idempotent $Q$, the matrix $(p-1).I_n-Q$ is never a sum of two idempotents
since it is diagonalizable with eigenvalues in $\{(p-1)1_\K,(p-2)1_\K\}$
(cf. Theorem \ref{HPsum}).

\item Assume $\K$ is not a prime field.
Let $\alpha \in \K \setminus \F_p$. Then the matrix $\alpha.I_p$
has trace $0$, and the same line of reasoning as in (i) shows that it is not
a sum of three idempotents.
\end{enumerate}

\section{Fields of characteristic $2$ or $3$}

\begin{prop}
Set $p:=\car(\K)$ and assume $p \in \{2,3\}$.
Then every matrix of $\Mat_n(\K)$ which is a sum of idempotents
is actually a sum of four idempotents.
\end{prop}

\begin{proof}
Let $A \in \Mat_n(\K)$ such that $\tr A \in \F_p$.
By reduction to a rational canonical form, we find that
$A \sim D\bigl(C(P_1),C(P_2),\dots,C(P_N),\alpha.I_q\bigr)$
for some monic polynomials $P_1,\dots,P_N$ of degree at least $2$, some $\alpha \in \K$ and some $q \in \N$.
We first study the case $q=p$ and $N=0$.
\begin{itemize}
\item Assume $\car(\K)=2$. Then $\alpha.I_2$ is a sum of four idempotent matrices:
indeed $\alpha.I_2-\begin{bmatrix}
1 & 0 \\
-1 & 0
\end{bmatrix}$ is cyclic, so it is a sum of three idempotents (according to point (ii) in the proof of Theorem \ref{sumof5}), whilst
$\begin{bmatrix}
1 & 0 \\
-1 & 0
\end{bmatrix}$ is idempotent.
\item Assume $\car(\K)=3$. We then contend that $\alpha.I_3$ is a sum of four idempotents:
setting $\beta:=\alpha-2$, it suffices to prove that $\beta.I_3$ is a $(1,-1,1,-1)$-composite.
Indeed, we know that $D(0,\beta,-\beta)$ is a difference of two idempotents, and
$\beta.I_3-D(0,\beta,-\beta)=D(\beta,0,-\beta)$ is also a difference of two idempotents,
which proves our claim\footnote{More generally, for any field $\K$ of characteristic $p > 0$,
and every $\alpha \in \K$, the scalar
matrix $\alpha.I_p$ is a sum of four idempotents.}.
\end{itemize}
In any case, we may reduce the study to the case
$q \in \{0,1,2\}$ by ``moding out" the $\alpha.I_p$ blocks
(notice that the trace is unaltered by doing so).
From now on, we will assume $q \in \{0,1,2\}$.

\vskip 2mm
\noindent If then $q=2$, we write $A \sim D(\alpha,C(P_1),C(P_2),\dots,C(P_N),\alpha)$.
In any case, we have found non-constant monic polynomials
$R_1,\dots,R_M$ such that $\deg R_i \geq 2$ for all $i \in \lcro 2,M-1\rcro$ and
$$A \sim A':=D\bigl(C(R_1),C(R_2),\dots,C(R_M)\bigr).$$
It will thus suffice to prove that $A'$ is a sum of four idempotents. \\
For every $k \in \lcro 1,M\rcro$, set $n_k:=\deg R_k$, and define
$$Q:=\begin{bmatrix}
F_{n_1} & 0 & \cdots & 0 \\
-H_{n_2,n_1} & F_{n_2} & \ddots & \vdots \\
0 & \ddots & \ddots & 0  \\
 & 0 & -H_{n_N,n_{N-1}} & F_{n_N}
 \end{bmatrix}.$$
Then $Q$ is idempotent and $A-Q$ is a good cyclic matrix
with trace in $\F_p$. We write
$A-Q=\begin{bmatrix}
B_1 & C \\
H_{1,n-1} & B_2
\end{bmatrix}$ with $B_1 \in \Mat_{n-1}(\K)$, $C \in \Mat_{n-1,1}(\K)$ and $B_2 \in \K$. \\
If $p=2$, then the proof from Proposition \ref{sumof5} shows that $A-Q$ is a sum of three idempotents. \\
Assume finally that $p=3$ and set $\delta:=\tr(A-Q)-(n+1).1_\K$. By Lemma \ref{cyclicfit}, there
exists a column matrix $C' \in \Mat_{n-1,1}(\K)$ such that
$$\begin{bmatrix}
B_1 & C' \\
H_{1,n-1} & B_2-(n+1).1_\K
\end{bmatrix} \sim C(X^{n-1}(X-\delta)).$$
Since $\delta \in \{0,1,-1\}$, Theorems \ref{HPsum} and \ref{HPsumcar2} then show that
$$A-Q-I_n-\begin{bmatrix}
0 & C-C' \\
0 & 1
\end{bmatrix}=\begin{bmatrix}
B_1 & C' \\
H_{1,n-1} & B_2-(n+1).1_\K
\end{bmatrix}$$
is a difference of two idempotents, hence $A$ is a sum of four idempotents.
\end{proof}

\section{A lower asymptotic upper bound for prime fields}

In this final part, we will prove that for a prime field, the asymptotic bound of five idempotents from Theorem
\ref{sumof5} can actually be lowered to four.

\begin{theo}\label{asymptoticsum4}
Assume $\K=\F_p$ for some prime $p$. \\
Then there exists an integer $n_0$ such that, for every
$n \geq n_0$, any matrix of $\Mat_n(\K)$ is a sum of $4$ idempotents.
\end{theo}

It will of course suffice to prove that, for some integer $n_0$, any matrix of $\Mat_n(\K)$ with $n \geq n_0$
is a $(1,-1,1,-1)$-composite.

\noindent
We start by tackling the case of scalar matrices:

\begin{lemme}\label{scalarinto4}
There exists an integer $n_0$ such that,
for every $\alpha \in \F_p$ and every integer $n \geq n_0$, the matrix
$\alpha.I_n$ is a $(1,-1,1,-1)$-composite.
\end{lemme}

\begin{proof}
Let $\alpha \in \F_p$.
Since $\F_p$ is finite, it will suffice to prove that $\alpha.I_n$ is a $(1,-1,1,-1)$-composite for large enough $n$. \\
If $\alpha \in \{-1,0,1\}$, then the result is trivial.
Assume now this is not the case (and so $p \geq 5$).
We wish to prove that, for large enough $n$, there is a diagonal matrix
$D$ such that both $D-\alpha.I_n$ and $D$ are differences of idempotents.
Let then $D$ be an arbitrary diagonal matrix, and, for every $\lambda \in \K$,
set $n(\lambda):=\dim \Ker(D-\lambda.I_n)$. \\
Theorem \ref{HPdiff} then shows that for $D$ to satisfy the previous conditions, it is sufficient (and necessary) that:
\begin{enumerate}[(i)]
\item $n(-\lambda)=n(\lambda)$ for every $\lambda \in \F_p \setminus \{0,1,-1\}$;
\item $n(\lambda)=n(2\alpha-\lambda)$ for every $\lambda \in \F_p \setminus \{\alpha,\alpha+1,\alpha-1\}$.
\end{enumerate}
Our lemma will thus be proven if we show that,
for every large enough $n$, there is a family
$(a_k)_{k \in \F_p}$ of non-negative integers such that:
\begin{enumerate}[(i)]
\item $a_{-k}=a_k$ for every $k \in \F_p \setminus \{0,1,-1\}$;
\item $a_k=a_{2\alpha-k}$ for every $k \in \F_p \setminus \{\alpha,\alpha-1,\alpha+1\}$;
\item $\underset{k\in \F_p}{\sum}a_k=n$.
\end{enumerate}
Consider the two involutions $\sigma : k \mapsto -k$ and $\tau : k \mapsto 2\,\alpha-k$ of $\F_p$.
Let $\calR$ denote the equivalence relation on $\F_p$ generated by the two sets of elementary relations:
\begin{itemize}
\item $\forall k \in \F_p \setminus \{1,-1\}, \; \sigma(k) \sim k$;
\item $\forall k \in \F_p \setminus \{\alpha+1,\alpha-1\}, \; \tau(k) \sim k$.
\end{itemize}
We wish to show that $\calR$ is non-trivial relation, i.e. that it has at least two classes.
Clearly, we can pick two distinct elements $a$ and $b$ in the set $\{1,-1,\alpha+1,\alpha-1\} \setminus \{1\}$.
Assume $1\calR a$ and $1 \calR b$. \\
Then there are two minimal chains $1=a_0 \sim a_1 \sim a_2 \sim \dots \sim a_r=a$
and $1=b_0 \sim b_1 \sim b_2 \sim \dots \sim b_s=b$.
Since $\sigma$ and $\tau$ are involutions, an easy induction proves that
$a_{i+1}=\tau(a_i)$ and $b_{i+1}=\tau(b_i)$ for every even $i$, and
$a_{i+1}=\sigma(a_i)$ and $b_{i+1}=\sigma(b_i)$ for every odd $i$. It follows that
$a_i \not\in \{1,-1,\alpha+1,\alpha-1\}$ and $b_j \not\in \{1,-1,\alpha+1,\alpha-1\}$
for any $(i,j) \in \lcro 1,r-1\rcro \times \lcro 1,s-1\rcro$. Hence the two
previous chains are equal, which leads to the contradiction $a=b$.

\vskip 2mm
\noindent The previous \emph{reductio ad absurdum} proves that there are at least two classes
for the equivalence relation $\calR$. \\
Therefore, all the integers\footnote{Of course, $\# x$ denotes here the cardinal of the equivalence class $x$.}
$\# x$, for $x \in \F_p/\calR$,  belong to $\lcro 1,p-1\rcro$, and
since their sum is the prime $p$, they are globally mutually prime (i.e. their
greatest common divisor is $1$).
Since $\F_p$ has only finitely many partitions, Lemma \ref{scalarinto4} can be deduced from the classic
lemma of number theory that follows.
\end{proof}

\begin{lemme}
Let $a_1,\dots,a_r$ be positive integers that are globally mutually prime.
Then there exists a positive integer $N$ such that
$$\forall n \geq N, \; \exists (b_1,\dots,b_r)\in \N^r:  n=\underset{k=1}{\overset{r}{\sum}}\,b_r\,a_r.$$
\end{lemme}

\noindent We now move on to the second key lemma:

\begin{lemme}\label{tracefitsum}
Let $P_1,\dots,P_s$ denote non-constant monic polynomials of $\K[X]$.
Set $n_k:=\deg P_k$ for all $k$, then $N:=\underset{k=1}{\overset{s}{\sum}}n_k$ and
 $A:=D\bigl(C(P_1),\dots,C(P_s)) \in \Mat_N(\K)$. \\
Assume $\deg P_s \geq 2$. Then, for every integer $r \in \lcro s+1,N+1\rcro$
and for every monic polynomial $P$ of degree $N$ and trace $\tr(A)-r.1_\K$, there
are two idempotents $Q$ and $Q'$ of $\Mat_N(\K)$ such that $A-Q-Q' \sim C(P)$.
\end{lemme}

\begin{proof}
For every $k \in \lcro 1,s-1\rcro$, we choose arbitrarily two column matrices
$C_k$ and $C'_k$ in $\Mat_{n_k,1}(\K)$ and a diagonal matrix $D_k=D(1,\dots,1,0,\dots,0) \in \Mat_{n_k}(\K)$ with last coefficient $0$.
We also choose arbitrarily two column matrices $C_s$ and $C'_s$ in $\Mat_{n_s-1,1}(\K)$ and a diagonal matrix $D_s \in \Mat_{n_s-1}(\K)$
with coefficients in $\{0,1\}$.
We set
$$B=\begin{bmatrix}
D_1+F_{n_1} & 0 & & & & \\
-H_{n_2,n_1} & 0 & 0 & & & \\
0 & 0 & D_3+F_{n_3} & \ddots  \\
\vdots & & & \ddots & &
\end{bmatrix}
\quad ; \quad
B'=
\begin{bmatrix}
0 & 0 & 0 & & & \\
0 & D_2+F_{n_2} & 0 & & & \\
0 & -H_{n_3,n_2} & 0 & 0  \\
0 & 0 & 0 & D_4+F_{n_4}  \\
0 & 0 & 0 & -H_{n_5,n_4} & \ddots  \\
\vdots & & & & \ddots & &
\end{bmatrix}$$
$$C=\begin{bmatrix}
C_1 \\
C_2 \\
\vdots \\
C_s
\end{bmatrix} \quad ; \quad
C'=\begin{bmatrix}
C'_1 \\
C'_2 \\
\vdots \\
C'_s
\end{bmatrix}.$$
Finally, we set
$$Q=\begin{bmatrix}
B & C \\
0 & 1
\end{bmatrix} \quad \text{and} \quad Q'=\begin{bmatrix}
B & C' \\
0 & 1
\end{bmatrix}.$$
Straightforward computation shows that the matrices $Q$ and $Q'$ are both idempotents provided the following conditions
hold:
\begin{enumerate}[(i)]
\item $C_k=0$ for every odd integer $k$;
\item $C'_k=0$ for every even integer $k$.
\end{enumerate}
We now choose an arbitrary column matrix $C_0 \in \Mat_{N-1,1}(\K)$ : the
 $C_k$'s and $C'_k$'s can be chosen so as to satisfy the previous conditions
together with $C_0=C+C'$, and we choose them accordingly.
Hence $Q$ and $Q'$ are idempotents, and
$$A-Q-Q'=\begin{bmatrix}
B_1 & C_0 \\
H_{1,N-1} & ?
\end{bmatrix}$$
for some good cyclic matrix $B_1$ which depends only on the choice of $D_1,\dots,D_s$.
Hence Lemma \ref{cyclicfit} shows that for every monic polynomial $P$ of degree $N$ and trace
$\tr(A-Q-Q')$, we can choose $C_0$ such that
$$A-Q-Q' \sim C(P).$$
To conclude, we simply remark that
$$\tr(A-Q-Q')=\tr(A)-(s+1)-\underset{k=1}{\overset{s}{\sum}}\tr D_k$$
and that any element of $\bigl\{k.1_\K  \mid 0 \leq k \leq N-s\bigr\}$ can be reached by
$\underset{k=1}{\overset{s}{\sum}}\tr D_k$ if the $D_k$'s are carefully chosen.
\end{proof}

\begin{cor}\label{tracefitdiff}
With the assumptions from Lemma \ref{tracefitsum},
for every integer $r \in \lcro s+1-N,1\rcro$
and every monic polynomial $P$ of degree $N$ and trace $\tr(A)-r.1_\K$, there
are idempotents $Q$ and $Q'$ of $\Mat_N(\K)$ such that $A-(Q-Q') \sim C(P)$.
\end{cor}

\begin{proof}
It suffices to apply Lemma \ref{tracefitsum} to the matrix $A'=A+I_N$.
\end{proof}

\noindent Finally, we will need the following lemma:

\begin{lemme}[Embedding lemma]\label{nilfit4}
Let $r \geq 2$ and $P$ be a monic polynomial of degree $r$.
Then there is an integer $m_r$, depending only on $r$, such that the matrix
$\begin{bmatrix}
C(P) & 0 \\
0 & 0_{m_r}
\end{bmatrix} \in \Mat_{r+m_r}(\K)$ is a $(1,-1,1,-1)$-composite.
\end{lemme}

\begin{proof}
To start with, let us remark that if $m_r$ is a solution, any integer greater that $m_r$ is also a solution.
We first choose an integer $k \in \lcro 0,p-1\rcro$ such that
$\tr(P)-(r+k+1).1_\K=0$. Corollary \ref{tracefitdiff} then provides idempotents
$Q_1$ and $Q_2$ in $\Mat_{r+k}(\K)$ such that
$$\begin{bmatrix}
C(P) & 0 \\
0 & 0_k
\end{bmatrix}-(Q_1-Q_2) \sim C((X-1)^{r+k}).$$
Consider then the block-diagonal matrix
$$B:=D\bigl((X+1)^{r+k},(X-1)^{r+k-1},(X+1)^{r+k-1},(X-1)^{r+k-2},\dots,(X-1),(X+1)\bigr).$$
Theorems \ref{HPdiff} and \ref{HPsumcar2} ensure that $B$ is a difference of two idempotents $Q'_1$ and $Q'_2$.
Letting $N$ denote the size of $B$, we obtain
$$\begin{bmatrix}
C(P) & 0 \\
0 & 0_{N+k}
\end{bmatrix}
-\begin{bmatrix}
Q_1 & 0 \\
0 & Q'_2
\end{bmatrix}+\begin{bmatrix}
Q_2 & 0 \\
0 & Q'_1
\end{bmatrix} \sim D\bigl(C((X-1)^{r+k}),B\bigr).$$
Another use of Theorems \ref{HPdiff} and \ref{HPsumcar2} proves then that
this last matrix is itself a $(1,-1)$-composite.
Roughly $N+k\leq p+(p+r)+2\frac{(p+r-1)\,(p+r)}{2}$,
so the integer $m_r=2\,p+r+(p+r)^2$ is a solution.
\end{proof}

\vskip 3mm
\noindent We are now ready to prove Theorem \ref{asymptoticsum4}.

\noindent Let $A \in \Mat_n(\K)$. By reduction to the rational canonical form, we find an $\alpha \in \K$, an integer $q \geq 0$ and
monic polynomials $P_1,\dots,P_s$ of degree greater or equal to $2$ such that
$$A \sim C\bigl(\alpha.I_q,C(P_1),\dots,C(P_s)\bigr).$$
Set $N:=n-q=\underset{k=1}{\overset{s}{\sum}}\,\deg P_k$.

\noindent We wish to prove that, provided $n$ is large enough, $A$ is automatically a $(1,-1,1,-1)$-composite.
Lemma \ref{scalarinto4} already provides an integer $n_0$ such that
$\beta.I_m$ is a $(1,-1,1,-1)$-composite for every $\beta \in \F_p$ and every integer $m \geq n_0$.

\begin{itemize}
\item Assume first $N\geq 2p$. Then $N-s \geq \frac{N}{2}\geq p$.
Hence $\tr (A) \equiv r$ mod. $p$ for some $r \in \lcro s-N+1,1\rcro$,
and Corollary \ref{tracefitdiff} provides idempotents $Q_1$ and $Q_2$ such that
$A-(Q_1-Q_2) \sim C(X^n)$, so $A-(Q_1-Q_2)$ is nilpotent and itself a difference of idempotents.
\vskip 2mm
\item Assume $N<2\,p$, $q \geq p+n_0$ and $\alpha \neq 0$.
We write
$$A \sim \begin{bmatrix}
\alpha.I_q & 0 \\
0 & A_1
\end{bmatrix} \quad \text{with} \quad A_1 \sim D\bigl(C(P_1),\dots,C(P_s)\bigr).$$
Since $\alpha \neq 0$, we have $\tr A_1-1+t.\alpha=0$ for some $t \in \lcro 0,p\rcro$.
Decompose then
$$A \sim \begin{bmatrix}
\alpha.I_{q-t} & 0 \\
0 & A_2
\end{bmatrix} \quad \text{with} \quad A_2=\begin{bmatrix}
\alpha.I_t & 0 \\
0 & A_1
\end{bmatrix}.$$
Corollary \ref{tracefitdiff} provides idempotents $Q$ and $Q'$ such that $A_2-(Q-Q') \sim C(X^{N-t})$
so $A_2$ is a $(1,-1,1,-1)$-composite. \\
Since $q-t \geq n_0$, we learn that $\alpha.I_{q-t}$ is also a $(1,-1,1,-1)$-composite. \\
It then follows that $A$ is itself a $(1,-1,1,-1)$-composite.
\item Assume finally that $N<2\,p$ and $\alpha=0$.
Choose, for every integer $r \geq 2$, an integer $m_r$ provided by Lemma \ref{nilfit4}.
Assume $q \geq p\,\underset{2 \leq r<2\,p}{\max}\,m_r$.
We can then decompose
$$A \sim D(A_1,\dots,A_s,0,\dots,0)$$
with, for every $k \in \lcro 1,s\rcro$,
$$A_k \sim \begin{bmatrix}
C(P_k) & 0 \\
0 & 0_{m_k}
\end{bmatrix}.$$
By Lemma \ref{nilfit4}, every $A_k$ is a $(1,-1,1,-1)$-composite, so $A$ also is.
\end{itemize}

\noindent Finally, provided $n$ is large enough\footnote{Say $n>2\,p+\max\biggl[p+n_0,\,p\,\underset{2 \leq r<2\,p}{\max}\,m_r\biggr]$.},
then $A$ automatically falls into one of the three categories we have just inspected. This finishes our proof of Theorem \ref{asymptoticsum4}.

\begin{Rem}
Whether this upper bound of $4$ idempotents still holds for an arbitrary non-prime field of positive characteristic
remains an open problem so far.
\end{Rem}

\section*{Acknowledgements}
I would like to express deep gratitude towards V. Rabanovich for submitting me some of the questions tackled here.

\end{document}